\documentclass{amsart}

\usepackage{hyperref}

\usepackage{amssymb, latexsym,pdfsync}
\theoremstyle{plain}
\newtheorem{theorem}{Theorem}
\newtheorem{corollary}{Corollary}
\newtheorem{lemma}{Lemma}

\theoremstyle{definition}

\newtheorem{remark}{Remark}

\newcommand{\CC}{\mathbb{C}}
\newcommand{\FF}{\mathbb{F}}
\newcommand{\Fp}{\mathbb{F}_p}
\newcommand{\Fq}{\mathbb{F}_q}

\newcommand{\Fqt}{\mathbb{F}_{q^2}}

\newcommand{\PPP}{\mathcal P}

\def\PG{\mathrm{PG}}

\newcommand{\npmatrix}[1]{\left[ \begin{matrix} #1 \end{matrix} \right]}

\newcommand{\rank}{\mathrm{rank}}

\begin{document}

\title[Constant rank-distance sets]
{Constant rank-distance sets of hermitian matrices and partial spreads in hermitian polar spaces}
\author{Rod Gow, Michel Lavrauw, John Sheekey and Fr\'ed\'eric Vanhove}
\address{
Rod Gow: Mathematics Department\\
University College\\
Belfield, Dublin 4\\
Ireland
\newline Michel Lavrauw: Department of Management and Engineering\\
Universit\`a di Padova\\
Italy
\newline John Sheekey: Department of Management and Engineering\\
Universit\`a di Padova\\
Italy
\newline Fr\'ed\'eric Vanhove: Department of Mathematics\\ Ghent University\\Belgium
}
\email{rod.gow@ucd.ie; michel.lavrauw@unipd.it; johnsheekey@gmail.com;fvanhove@cage.ugent.be}
\begin{abstract}

In this paper we investigate partial spreads of $H(2n-1,q^2)$ through the related notion of partial spread sets of hermitian matrices, and the more general notion of constant rank-distance sets. We prove a tight upper bound on the maximum size of a linear constant rank-distance set of hermitian matrices over finite fields, and as a consequence prove the maximality of extensions of symplectic semifield spreads as partial spreads of $H(2n-1,q^2)$. We prove upper bounds for constant rank-distance sets for even rank, construct large examples of these, and construct maximal partial spreads of $H(3,q^2)$ for a range of sizes.

\end{abstract}
\maketitle

\section{Introduction}
\label{sect:intro}

A \emph{partial $(t-1)$-spread} of a projective or polar space $\PPP$ is a set $S$ of pairwise disjoint $(t-1)$-dimensional subspaces of $\PPP$. A partial spread is called a \emph{spread} if the elements of $S$ cover $\PPP$. A $(t-1)$-spread of $\PG(N-1,q)$ exists if and only if $t$ divides $N$ (\cite{Segre64}).

A \emph{partial spread set} $U$ is a set of $n \times n$ matrices over a field $F$ such that 
\begin{align*}
(i)~\rank(A-B) = n & ~\textrm{for all $A,B \in U$, $A \ne B$};\\
(ii)~\rank(A) = n & ~\textrm{for all $A \in U$, $A \ne 0$}.
\end{align*}
We will denote the space of $n\times n$ matrices over $F$ by $M_n(F)$, and the spaces of hermitian and symmetric $n\times n$ matrices over $F$ by $H_n(F)$ and $S_n(F)$, respectively.

It is well known that a partial spread set in $M_n(F)$ defines a partial spread in the projective space $\PG(2n-1,F)$; a partial spread set in $H_n(K)$ defines a partial spread in the hermitian polar space $H(2n-1,K)$; and a partial spread set in $S_n(F)$ defines a partial spread in the symplectic polar space $W(2n-1,F)$. We will provide further explanation of this in Section \ref{sect:pre}.

In this paper we investigate partial spread sets as a special case of the following more general definition:

A \emph{constant rank-distance $k$ set} is a set of $n \times n$ matrices $U$ such that
\begin{align*}
(i)~\rank(A-B) = k & ~\textrm{for all $A,B \in U$, $A \ne B$};\\
(ii)~\rank(A) = k & ~\textrm{for all $A \in U$, $A \ne 0$}.
\end{align*}
This name follows the definition of \emph{rank-distance} in \cite{Gab1985}. Clearly a constant rank-distance $n$ set is a partial spread set. Note that a set satisfying only property $(i)$ implies the existence of a constant rank-distance $k$ set of the same size. For suppose $U'$ is such a set, and choose some $A \in U'$. Then it is easily verified that $U := \{A-B~:~B \in U'\}$ is a constant rank-distance $k$ set.

A constant rank-distance $k$ set (resp. partial spread set) is said to be \emph{maximal} if it is not strictly contained in a larger constant rank-distance $k$ set (resp. partial spread set).

If a constant rank-distance $k$ set (resp. partial spread set) is closed under $F'$-linear combinations for some field $F'$, we refer to it as an {\em $F'$-linear constant rank-distance $k$ set} (resp. partial spread set). We will often omit the specification of a particular field.

In Section \ref{sect:char}, we will prove a new bound for linear constant rank-distance sets of hermitian matrices over finite fields, using properties of characters defined on function spaces. We will treat constant-rank distance sets in general, using the theory of association schemes, in Section \ref{sec:assoc}. In Section \ref{sect:mps}, we will use the results of Section \ref{sect:char} to prove new results on maximal partial spreads of $H(2n-1,q^2)$. We will also construct maximal partial spreads of $H(3,q^2)$ of all sizes in the range $[q^2+1,q^2+q]$. Finally in Section \ref{sect:crs}, we will construct some constant rank-distance sets which are larger than the largest possible linear constant rank-distance sets, including partial spread sets larger than the largest possible linear partial spread sets.

\section{Partial spreads and subspace codes}
\label{sect:pre}

We recall now the connection between partial spreads and partial spread sets (\cite[p. 220]{Dembowski}), and the definition of a subspace code. Given an $n \times n$ matrix $A$, we can define an $(n-1)$-dimensional subspace $S_A$ of $\PG(2n-1,q)$ as follows:
\[
S_A := \langle \npmatrix{u\\Au}~:~u \in \Fq^n\setminus\{0\}  \rangle
\]
Given any two matrices $A$ and $B$, it is easy to see that
\[
\dim(S_A \cap S_B) = n-\rank(A-B)-1,
\]
and hence $S_A \cap S_B = \emptyset$ if and only if $\rank(A-B)=n$.
We define also the $(n-1)$-dimensional subspace $S_{\infty}$ of $\PG(2n-1,q)$ by
\[
S_{\infty} := \langle \npmatrix{0\\u}~:~u \in \Fq^n \setminus\{0\} \rangle.
\]
Again it is clear that $S_A \cap S_{\infty} = \emptyset$ for all $A \in M_n(\Fq)$, and the subspaces $S_A$ are precisely the $(n-1)$-spaces which are disjoint from $S_{\infty}$. 

Hence if $U$ is a partial spread set in $M_n(\Fq)$, then the set
\[
D_U := \{S_A~:~A \in U\}\cup \{S_{\infty}\}
\]
is a partial spread of $\mathrm{PG}(2 n -1,q)$, with $|D_U| = |U|+1$. If $D_U$ is a spread, then $U$ is called a \emph{spread set}. A spread set in $M_n(\Fq)$ has size $q^n$. Conversely, any partial spread $D$ of size at least two is equivalent to a partial spread containing both $S_0$ and $S_{\infty}$, and hence is equivalent to $D_U$ for some partial spread set $U$. 

Every spread $D$ of $\PG(2n-1,q)$ defines a \emph{translation plane}, via the Andr\'e-Bruck-Bose construction. If the dual of this plane is also a translation plane, it is called a \emph{semifield plane}. Such a plane can be coordinatized by a \emph{semifield} (see Chapters 5 and 6 of \cite{HuPi1973} for further details on coordinatization of a projective plane). This occurs if and only if the spread $D$ corresponds to some spread set which is linear over a subfield of $\Fq$. For this reason, a linear spread set is also called a \emph{semifield spread set}. For more on semifields we refer to \cite{LaPo2011}.

A \emph{hermitian polar space} $H(t-1,q^2)$ is the geometry of subspaces of $\PG(t-1,q^2)$ which are totally isotropic with respect to some non-degenerate hermitian form on $\Fqt^t$. It is well known that the maximum dimension of a subspace contained in $H(t-1,q^2)$ is $\lfloor \frac{t}{2} \rfloor -1$.

The projective geometry $\PG(t-1,q^2)$ contains $\PG(t-1,q)$ as a subgeometry. A \emph{symplectic polar space} $W(t-1,q)$ is the geometry of subspaces in $\PG(t-1,q)$ which are isotropic with respect to some non-degenerate symplectic form on $\Fq^t$. It is necessary that $t$ be even.

Consider now the space $H(2n-1,q^2)$. We will let $x\mapsto \overline{x}$ denote the Frobenius automorphism $x \mapsto x^q$.  For any matrix $A$, we will denote by $\overline{A}$ this automorphism applied entrywise. We take the non-degenerate hermitian form defined by the matrix
\[
\npmatrix{0_n&a I_n\\-a I_n&0_n},
\]
where $a\in \Fqt^{\times}$ is such that $\overline{a} = -a$. Then given a matrix $A \in M_n(\Fqt)$, the space $S_A$ is contained in $H(2n-1,q^2)$ if and only if it is hermitian, since
\[
\npmatrix{\overline{u}^T& \overline{u}^T \overline{A}^T} \npmatrix{0_n&a I_n\\-a I_n&0_n} \npmatrix{u\\Au} = 0
\]
holds for all $u \in \Fqt^n$ if and only if
\[
a \overline{u}^T (A - \overline{A}^T)u = 0
\]
for all $u \in \Fqt^n$, and hence if and only if $A = \overline{A}^T$, as claimed. 

Hence a partial spread set $U$ in $H_n(\Fqt)$ leads to a partial spread $D_U$ in $H(2n-1,q^2)$, with $|D_U| = |U|+1$. Conversely, it is well known (and follows from the discussion at the beginning of this section) that every partial spread in $\PG(2n-1,q)$ is equivalent to $D_U$ for some partial spread set $U$ in $M_n(\Fq)$ \cite{DembOs}, and hence every partial spread in $H(2n-1,q^2)$ is equivalent to $D_U$ for some partial spread set $U$ in $H_n(\Fqt)$.

Clearly the intersection of $H(2n-1,q^2)$ with $\PG(2n-1,q)$ is a symplectic polar space $W(2n-1,q)$, with symplectic form defined by the matrix
\[
\npmatrix{0_n& I_n\\- I_n&0_n}.
\]
A partial spread in $W(2n-1,q)$ then leads to a partial spread set in $S_n(\Fq)$, which is of course an $\Fq$-subspace of $H_n(\Fqt)$. In fact it is also a partial spread set in $H_n(\Fqt)$, as the rank of a matrix does not change over an extension field. For any subspace in $W(2n-1,q)$, the subspace obtained by extending scalars from $\Fq$ to $\Fqt$ is a subspace in $H(2n-1,q^2)$. Given a partial spread $D$ of $W(2n-1,q)$, we refer to the partial spread of $H(2n-1,q^2)$ obtained by extending each subspace as the \emph{extension} of $D$. An element $A \in S_n(\Fq)$ defines both a subspace $S_A$ in $W(2n-1,q)$, and a subspace $S'_A$ in $H(2n-1,q^2)$. A partial spread set $U$ in $S_n(\Fq)$ hence defines both a partial spread $D_U$ of  $W(2n-1,q)$, and a partial spread $D'_U$ of $H(2n-1,q^2)$. It is clear that $S'_A$ is in fact the extension of $S_A$, and $D'_U$ is the extension of $D_U$.

Spreads exist in $W(2n-1,q)$ for all $q,n$, and have size $q^n+1$. In fact \emph{semifield} spreads exist in $W(2n-1,q)$ for all $q,n$, as Kantor \cite{Kantor03} showed that a semifield spread in $\PG(2n-1,q)$ is symplectic if and only if the semifield it defines is Knuth-equivalent to a \emph{commutative semifield}. Such a spread is called a \emph{symplectic semifield spread}, see \cite{LaPo2011}.

Much study has been dedicated to the maximum size of a partial spread in $H(2n-1,q^2)$, or equivalently the maximum size of a partial spread set in $H_n(\Fqt)$. Particular attention has been paid to the case $H(3,q^2)$, or equivalently $H_2(\Fqt)$. See \cite{DeBKlMet12} for an overview of the known results.

A \emph{subspace code} is a set of subspaces of $\PG(N,q)$ together with the distance function $d(S,T) = \dim(\langle S,T \rangle)-\dim(S\cap T)$. A subspace code $C$ such that $\dim(S)=t-1$ for all $S \in C$ is called a \emph{constant dimension $t$ code} (note that $t$ is the vectorial dimension of the elements of the code, to follow the convention in the literature, where subspace codes are usually considered as subsets of the set of subspaces of a vector space). Subspace codes and constant dimension codes are of interest in network coding, see for example \cite{KhSiKs09} for a survey, and the connection with sets of matrices with particular rank properties (``rank metric codes'') has been studied in for example \cite{SiKsKo08}. 

Given a constant rank-distance $k$ set $U$ of $n \times n$ matrices, define the subset $D'_U = \{S_A~:~A \in U\}\subseteq \PG(2n-1,F)$, where $S_A$ is defined as above. Then $D'_U$ is a constant dimension $n$ code. Moreover, $d(S_A,S_B) = 2k$ for all $A,B \in U$. Hence $D'_U$ is also a \emph{constant distance code}, and $|D'_U| = |U|$. Note that if $k<n$ we do not include the space $S_{\infty}$, as $d(S_A,S_{\infty}) = 2n \neq 2k$. 

When $k=n$ we can include $S_{\infty}$, and $D_U = D'_U \cup \{S_{\infty}\}$ is a partial spread, or equivalently a constant dimension $n$, constant distance $2n$ code. If $C$ is also a spread, it is more commonly referred to as a \emph{spread code}. See for example \cite{MaGoRo08} for more on spread codes. While subspace codes are normally studied in $\PG(N,q)$, in this work we will focus on those in $H(2n-1,q^2)$.

We summarize this discussion in the following lemmas for future reference.
\begin{lemma}\label{lem:partialspreadcorrespondence}
There exists a partial spread set in $H_n(\Fqt)$ of size $N$ if and only if there exists a partial spread in $H(2n-1,q^2)$ of size $N+1$.
\end{lemma}
\begin{lemma}\label{lem:gendimcorrespondence}
If there exists a constant rank-distance $k$ set in $H_n(\Fqt)$ of size $N$, then there exists a constant dimension $n$, constant distance $2k$ code in $H(2n-1,q^2)$ of size $N$.
\end{lemma}

\section{Character theory and constant rank-distance sets}
\label{sect:char}

We will prove a new upper bound on linear constant rank-distance sets by considering the set of hermitian matrices as an additive group, and investigating particular characters on this group. Throughout the rest of this paper we will assume $q=p^e$ for some prime $p$ and positive integer $e$.

We denote $W = \Fqt^n$. We will represent this space as column vectors; i.e. $W$ will be the vector space of $n \times 1$ matrices with entries in $\Fqt$.

Recall that $H_n(\Fqt)$ is a vector space over $\Fq$, but \emph{not} over $\Fqt$. To each hermitian matrix $h$ we associate a map from $W$ to $\Fq$ (which by abuse of notation we will also denote by $h$), by
\begin{align*}
h &:W \rightarrow \Fq \\
	&:w \mapsto \overline{w}^T h w.
\end{align*}
Note that $h(w)$ does indeed lie in $\Fq$ for all $w \in W$, as $\overline{h(w)} = \overline{h(w)^T} = \overline{w^T h^T \overline{w}} = \overline{w}^T \overline{h^T} w = \overline{w}^T h w = h(w)$.

For any $h \in H_n(\Fqt)$, and any $a \in \Fq$, we define the number
\[
N_h(a) := \#\{w \in W~|~h(w) = a\}.
\]

The proof of the following well-known lemma can also be found in \cite{JSThesis}.
\begin{lemma}\label{lem:Nha}
For any $h \in H_n(\Fqt)$, $\rank(h)=k$, we have
\[
N_{h}(a) = \left\{\begin{array}{ll}
								 q^{2n-k-1}(q^k+(-1)^k(q-1)) & \textrm{if $a=0$}\\
								 q^{2n-k-1}(q^k-(-1)^k) & \textrm{if $a \in \Fq^{\times}$} .
						 \end{array}\right.
\]
\end{lemma}

We now view $(H_n(\Fqt),+)$ as a finite group. Recall that a \emph{linear character} on a group $G$ is a homomorphism from $G$ into $\CC^{\times}$. A \emph{character} of an abelian group is a finite sum of linear characters. See for example \cite{Isaacs}. We can define some linear characters on $H_n(\Fqt)$ as follows. Let $\epsilon$ be a primitive $p$-th root of unity in $\CC$, and let $tr$ denote the trace map from $\Fq$ to $\Fp$. For each $w \in W$, define the function $\chi_{w}:H_n(\Fqt) \rightarrow \CC$ by
\[
\chi_{w} :h \mapsto \epsilon^{tr(h(\omega))}.
\]
It is clear that $\chi_{w}$ is a linear character, as $\chi_{w}(h+h') = \chi_{w}(h) \cdot \chi_{w}(h')$ for all $h,h' \in H_n(\Fqt)$.

Define now a character $\chi$ by
\[
\chi := \sum_{w \in W} \chi_{w}.
\]
Then by definition it is clear that
\[
\chi(h) = \sum_{w \in W} \epsilon^{tr(h(w))} = \sum_{a \in \Fq} N_h(a) \epsilon^{tr(a)}.
\]

This leads to the following lemma.

\begin{lemma}
\label{lem:chiher}
Suppose $h \in H_n(\Fqt)$ has rank $k$. Then
\[
\chi(h) = (-1)^k q^{2n-k}.
\]
\end{lemma}
\begin{proof}
By the above formula, we have $\chi(h) = \sum_{a \in \Fq} N_h(a) \epsilon^{tr(a)}$. But by Lemma \ref{lem:Nha}, $N_h(0)= q^{2n-k-1}(q^k+(-1)^k(q-1))$, and $N_h(a) = q^{2n-k-1}(q^k-(-1)^k)$ for all non-zero $a\in\Fq$, and so
\[
\chi(h) = q^{2n-k-1}(q^k+(-1)^k(q-1)) + q^{2n-k-1}(q^k-(-1)^k) \sum_{a \in \Fq^{\times}} \epsilon^{tr(a)}.
\]
But $\sum_{a \in \Fq^{\times}} \epsilon^{tr(a)} = -1$, giving us
\[
\chi(h) = q^{2n-k-1}(q^k+(-1)^k(q-1)) - q^{2n-k-1}(q^k-(-1)^k) = (-1)^k q^{2n-k},
\]
proving the claim.
\end{proof}

Recall the definition of the inner product of characters on any group $G$:
\[
\langle \phi,\psi\rangle :=\frac{1}{|G|} \sum_{g \in G} \phi(g)\overline{\psi(g)}.
\]
It is well known that this number is a non-negative integer \cite{Isaacs}. The restriction of a character on a group $G$ to a subgroup $H$ is again a character on $H$. We denote the trivial character by $1_G$, i.e. $1_G(g) = 1$ for all $g \in G$.

Suppose now we have a subgroup $U$ of $(H_n(\Fqt),+)$. Clearly $U$ is an $\Fp$-subspace of $H_n(\Fqt)$. We will use the restriction of the previously defined character $\chi$ on $(H_n(\Fqt),+)$ to $U$ to obtain upper bounds on the sizes of certain classes of subgroups.

\begin{lemma}
\label{lem:chihersp}
Let $U$ be an $\Fp$-subspace of $H_n(\Fqt)$, with $|U| = p^d$. Let $A_k$ denote the number of elements of $U$ of rank $k$. Then
\[
\sum_{k=0}^n  (-1)^k A_k p^{(2n-k)e - d}
\]
is a non-negative integer.
\end{lemma}
\begin{proof}
Consider the character $\chi$ on $H_n(\Fqt)$. Taking the inner product $\langle \chi|_U,1_U \rangle$ gives us
\[
\langle \chi|_U,1_U \rangle =\frac{1}{p^d} \sum_{h \in U} \chi(h).
\]
But by Lemma \ref{lem:chiher}, $\chi(h) = (-1)^k q^{2n-k} = (-1)^k p^{(2n-k)e}$ if $\rank(h)=k$, and so 
\[
\langle \chi|_U,1_U \rangle = \sum_{k=0}^n  (-1)^k A_k p^{(2n-k)e - d}
\]
is a non-negative integer, proving the result.
\end{proof}

This leads to the following.
\begin{theorem}
\label{thm:boundher}
Let $U$ be a linear constant rank-distance $k$ set of $H_n(\FF_{q^2})$, $k \ne 0$. Then
\[
|U| \leq \left\{ \begin{array}{cc}
q^k & \textrm{if $k$ is odd} \\
q^{2n-k} & \textrm{if $k$ is even}.
\end{array} \right.
\]
\end{theorem}

\begin{proof}
Consider $U$ as an $\Fp$-subspace of $H_n(\Fqt)$, and let $A_i$ be as in Lemma \ref{lem:chihersp}. If $|U|=p^d$ for some positive integer $d$, we have that $A_0 = 1$, $A_k = p^d-1$, $A_i=0$ otherwise. Hence by Lemma \ref{lem:chihersp}, the number
\[
p^{2ne-d}  + (-1)^k (p^d-1) p^{(2n-k)e - d} = p^{(2n-k)e - d} (p^{ke}+(-1)^k (p^d-1))
\]
must be a non-negative integer. Hence if $k$ is odd, we must have $p^{ke}\geq p^d$, and hence $|U|\leq q^k$, as claimed. If $k$ is even, then $p^{ke}+p^d-1$ is a positive integer relatively prime to $p$, and so $p^{(2n-k)e - d}$ must be an integer, implying $d \leq (2n-k)e$, and $|U| \leq p^{(2n-k)e} = q^{2n-k}$ as claimed.
\end{proof}

We will see in the next section that when $k$ is odd, this result holds for general constant rank-distance sets.

In \cite[Theorems 2 and 3]{DuGoSh2011} , Theorem \ref{thm:boundher} was proved for the special case where $U$ is an $\Fq$-subspace. Hence this theorem is a generalisation of that result. It was also shown that this bound is met in all cases. In the Sections \ref{sect:mps} and \ref{sect:crs} we will consider non-linear constant rank-distance sets which exceed these bounds. In the next section we will prove new upper bounds for general constant rank-distance sets, using the theory of association schemes.

\section{Association schemes}\label{sec:assoc}

An \emph{association scheme} $A$ is a finite set $\Omega$ together with a set of symmetric relations $R = \{R_0,\ldots,R_d\}$ such that
\begin{enumerate}
\item
$R_0$ is the identity relation,
\item
$R$ is a partition of $\Omega \times \Omega$,
\item
there exist non-negative integers $p_{ij}^k$ such that for any $(x,y) \in R_k$, the number of elements $z$ such that $(x,z) \in R_i$ and $(y,z) \in R_j$ is $p_{ij}^k$.
\end{enumerate}

Association schemes were introduced in \cite{BoSh}. We refer to \cite[\textsection 2]{BrCoNe} for proofs and more information.

For each association scheme there is a \emph{character matrix} or \emph{matrix of eigenvalues} $P$. The \emph{dual matrix of eigenvalues} is the matrix $Q = |\Omega|P^{-1}$.

For any non-empty subset $U$ of $A$, the \emph{inner distribution} $\textbf{a} = (\textbf{a}_0,\textbf{a}_1,\ldots,\textbf{a}_d)$ is defined by
\[
\textbf{a}_i = \frac{|(U \times U)\cap R_i|}{|U|}.
\]
Delsarte \cite{Del74} proved that any inner distribution $\textbf{a}$ must be such that $\textbf{a}Q$ has non-negative entries:
\begin{equation}\label{nonnegaQ}
(\textbf{a}Q)_j\geq 0,j\in\{0,\ldots,d\}.
\end{equation}

The sets of bilinear forms, alternating forms, and hermitian matrices each form an association scheme. The relations are defined by $(x,y) \in R_i \Leftrightarrow \rank(x-y) = i$ in the case of bilinear forms and hermitian matrices, and $(x,y) \in R_i \Leftrightarrow \rank(x-y) = 2i$ in the case of alternating forms (see \cite[\textsection 9.5]{BrCoNe}).

We note that the schemes of hermitian matrices in $H_n(\Fqt)$ have been characterized for $n\geq 3$ by their parameters $p_{ij}^k$ (see \cite{IvanovShpectorov82},\cite{IvanovShpectorov91} and \cite{Terwilliger95}).

In the association scheme of hermitian matrices, the inner distribution of a non-empty subset $U$ of $H_n(\FF_{q^2})$ is given by
\[\textbf{a}_i = \frac{|\{(a,b) \mid a,b \in U,\rank(a-b)=i\}|}{|U|}.
\]

The character matrices for the association schemes of bilinear and alternating forms were calculated in \cite{Del78} and \cite{DelGo75}, respectively.

For a hermitian matrix $X \in H_n(\FF_{q^2})$ define a character on $H_n(\FF_{q^2})$ by
\[
P_X(Y) := \epsilon^{tr(Tr(\overline{X}^T Y))}.
\]
Here $Tr$ denotes matrix trace, $tr$ denotes absolute field trace, and again $\epsilon$ is a primitive $p$-th root of unity in $\CC$.

One easily sees that $\sum_{\{X:\rank(X) = i\}} \epsilon^{tr(Tr(\overline{X}^T Y))}=\sum_{\{X:\rank(X) = i\}}\epsilon^{tr(Tr(\overline{Y}^T X))}$ only depends on $\rank(Y)$.  We can now define
\[
P_i := \sum_{\{X:\rank(X) = i\}} P_X,
\]
and
\[
P_i(j) = P_i(Y)
\]
where $Y$ is any hermitian matrix of rank $j$.  It now follows from \cite[\textsection 2.10.B]{BrCoNe} that the matrices $P$ and $Q$ satisfy:
\[P_{ji}=Q_{ji}=P_i(j).\]

Note that this character $P_i$ is related to the character $\chi$ from the preceding section in the following way. Every rank one hermitian matrix $X$ can be written as $u\overline{u}^T$ for some $u \in W$, and there are precisely $(q+1)$ vectors $u$ such that $X=u\overline{u}^T$. Now
\begin{align*}
P_X (A) 	&= \epsilon^{tr(Tr(\overline{X}^T A))}\\
						&= \epsilon^{tr(Tr(u\overline{u}^T A))}\\
						&= \epsilon^{tr(Tr( \overline{u}^T A u))}\\
						&= \epsilon^{tr(\overline{u}^T A u)}\\
						&= \chi_{u}(A).
\end{align*}
Hence we have that
\begin{equation}\label{eqn:chiP}
\chi = (q+1) P_1 +  P_0.
\end{equation}
Note that $P_0(A)=1$ for all $A$, i.e. $P_0$ is the trivial character. 

\begin{lemma}\label{lem:upperboundconstantrankdistance}
If $U$ is a constant rank-distance $k>0$ set in $H_n(\FF_{q^2})$ and $P_i(k)<0$, then $|U|\leq 1-\frac{P_i(0)}{P_i(k)}$.
\end{lemma}
\begin{proof}
We may assume that $U\neq\emptyset$. The assumptions on $U$ now imply that its inner distribution $\textbf{a}$ is given by:
\[
\textbf{a}_j = \left\{
\begin{array}{cc}
1 & \textrm{if $j=0$}\\
|U|-1 &  \textrm{if $j=k$}\\
0 & \textrm{otherwise}
\end{array}\right.
\]

and hence
\begin{align*}
(\textbf{a}Q)_i =(\textbf{a}P)_i &= \sum_j \textbf{a}_j P_i(j)\\
		&= P_i(0) + (|U|-1)P_i(k).
\end{align*}
If $P_i(k) < 0$ then by (\ref{nonnegaQ}):
\begin{equation}\label{eqn:genbound}
|U| \leq 1- \frac{P_i(0)}{P_i(k)}.
\end{equation}
\end{proof}

The value $P_j(0)$ equals the number of hermitian matrices of rank $j$ in $H_n(\Fqt)$, which is well known (see for instance \cite[p.127 and Theorem 9.5.7]{BrCoNe}) and given by
\begin{equation}\label{eqn:valencies}
P_j(0)={n \brack j}_{q^2} q^{\frac{j(j-1)}{2}} \prod_{i=1}^j (q^i+(-1)^i).
\end{equation}
\begin{remark}\label{rem:switch}
Note that since 
\begin{equation}\label{eqn:ratios}
\frac{P_i(k)}{P_i(0)}=\frac{P_k(i)}{P_k(0)}
\end{equation}
 (see for instance \cite[Lemma 2.2.1(iv)]{BrCoNe}) we also have, under the assumptions of Lemma \ref{lem:upperboundconstantrankdistance}, that
\begin{equation}\label{eqn:genbound2}
|U| \leq 1- \frac{P_k(0)}{P_k(i)}.
\end{equation}
So for every negative character value $P_i(j)$, we obtain an upper bound for both a constant rank-distance $i$ set and a constant rank-distance $j$ set.
\end{remark}

The values $P_i(j)$ were given in terms of Krawtchouk polynomials in \cite{Stanton}.

The values $P_1(j)$ follow from (\ref{eqn:chiP}) and Lemma \ref{lem:chiher}:
\begin{equation}\label{eqn:P1}
P_1(j) = \frac{(-q)^{2n-j}-1}{q+1}.
\end{equation}

Consequently, we have the following.
\begin{theorem}\label{thm:oddgen}
Let $U$ be any constant rank-distance $k$ set of $H_n(\FF_{q^2})$, $k$ odd. Then
\[
|U| \leq q^k.
\]
\end{theorem}

\begin{proof}
Applying Lemma \ref{lem:upperboundconstantrankdistance}, (\ref{eqn:valencies}) and (\ref{eqn:P1}), we obtain
\[
|U| \leq 1-\frac{P_1(0)}{P_1(k)}=1+\frac{q^{2n}-1}{q^{2n-k}+1} = (q^k+1)\left(1-\frac{1}{q^{2n-k}+1}\right) < q^k+1,
\]
and thus $|U|\leq q^k$, as claimed.
\end{proof}

Note that, in light of Lemma \ref{lem:gendimcorrespondence}, a slightly weaker upper bound of $q^k+1$ essentially follows from \cite[Lemma 3.2]{Van10}. 

We now calculate some further values, which will give new bounds in the even rank case.

\begin{lemma}\label{lem:formulaspecialeig}
For any $k$, we have
\[
P_k(n) = \prod_{i=1}^k \frac{(-q)^{i-1}-(-q)^n}{(-q)^i - 1}.
\]
\end{lemma}
\begin{proof}
For $k=0$, this is clear, and for $k=1$, this follows from (\ref{eqn:P1}). Suppose now that the formula holds for $P_0(n),\cdots,P_{k}(n)$, with $1\leq k\leq n-1$. The following identity then holds (see for instance \cite[p.128]{BrCoNe}):
\[
P_1(j)P_k(j)= c_{k+1}P_{k+1}(j) + a_kP_k(j)+b_{k-1}P_{k-1}(j),
\]
where
\begin{align*}
b_i &= \frac{q^{2 n}-q^{2 i}}{q+1}\\
c_i &= (-q)^{i-1}\frac{(-q)^i-1}{(-q)-1}\\
a_i &= b_0-b_i-c_i,
\end{align*}
for all $i\in\{0,\ldots,n\}$.

Hence:
\[ P_{k+1}(n)=\frac{1}{c_{k+1}} \left(\frac{1-(-q)^n}{-q-1}-a_k -b_{k-1}\frac{(-q)^k-1}{(-q)^{k-1}-(-q)^n}\right)P_k(n),\]
which, after some manipulation, can be simplified to
\[P_{k+1}(n)=\frac{(-q)^k-(-q)^n}{(-q)^{k+1}-1}P_k(n),\]
proving the claim.
\end{proof}

This leads immediately to the following.

\begin{theorem}
Let $U$ be any constant rank-distance $k$ set of $H_n(\FF_{q^2})$, $k \equiv 2 \mod 4$. Then
\[
|U| \leq 1+\prod_{i=1}^k (q^{n-i+1}+(-1)^{n-i+1}).
\]
\end{theorem}

\begin{proof}
By (\ref{eqn:genbound2}), if $P_k(n)$ is negative then
\[
|U| \leq 1-\frac{P_k(0)}{P_k(n)}.
\]
Now by (\ref{eqn:valencies})
\[
P_k(0) = q^{{k \choose 2}}\prod_{i=1}^k \frac{q^{2(n-i+1)} - 1}{q^{i} - (-1)^i},
\]
and by Lemma \ref{lem:formulaspecialeig}
\[
P_k(n) = (-q)^{{k \choose 2}}(-1)^{nk} \prod_{i=1}^k \frac{q^{n-i+1} - (-1)^{n-i+1}}{q^i-(-1)^i}.
\]
Hence
\[
\frac{P_k(0)}{P_k(n)} = (-1)^{nk+{k \choose 2}}\prod_{i=1}^k (q^{n-i+1}+(-1)^{n-i+1}),
\]

which is negative if $k \equiv 2 \mod 4$, completing the claim.
\end{proof}

\begin{remark}\label{rem:thasnew}
When $k=2$, the above gives an upper bound on a constant rank-distance $2$ set, that is
\begin{equation}\label{eqn:thasnew}
|U| \leq  q^{2n-1}+(-1)^{n-1}(q^n-q^{n-1}).
\end{equation}

By Remark \ref{rem:switch}, this is also an upper bound for a constant rank-distance $n$ set, that is, a partial spread set. In fact, this is exactly the bound of Thas \cite[Theorem 21]{Thas92} on partial spreads for even $n$ (translated to a bound on partial spread sets via Lemma \ref{lem:partialspreadcorrespondence}). We will discuss partial spread sets further in Section \ref{sect:mps}.
\end{remark}

The following theorem shows that the upper bound for constant rank-distance $2$ sets, is in fact an upper bound on sets under somewhat weaker assumptions as well.

\begin{theorem}Let $U$ be a subset of $H_n(\Fqt)$ for odd $n\geq 3$, such that for any $x,y\in U: \rank(x-y)\leq 2$.  Then $|U|\leq q^{2 n -1}+q^n-q^{n-1}$, and if equality holds, then $\{A-B:B \in U\}$ is a constant rank-distance $2$ set for all $A\in U$.
\end{theorem}
\begin{proof}Suppose $U$ is non-empty, and let $\textbf{a}$ be the inner distribution of $U$. By our assumption, $\textbf{a}_i=0$ if $i>2$.  Now by (\ref{nonnegaQ}) we have:
\[\textbf{a}_0 Q_{0 n}+\textbf{a}_1 Q_{1 n}+\textbf{a}_2 Q_{2 n}\geq 0,\]
or equivalently, since $P=Q$:
\[ 1 + \frac{P_n(1)}{P_n(0)} \textbf{a}_1+\frac{P_n(2)}{P_n(0)}\textbf{a}_2\geq 0,\]
which by (\ref{eqn:ratios}) can be rewritten  as
\[ 1 + \frac{P_1(n)}{P_1(0)} \textbf{a}_1+\frac{P_2(n)}{P_2(0)}\textbf{a}_2\geq 0.\]

It now follows from (\ref{eqn:valencies}) and Lemma \ref{lem:formulaspecialeig} that:
\[1+\frac{\textbf{a}_1}{(-q)^n+1}+\frac{\textbf{a}_2}{((-q)^{n-1}+1)((-q)^n+1)}\geq 0.\]
Since $n$ is odd and $\textbf{a}_2=|U|-\textbf{a}_1-1$, this yields:
\[|U|\leq q^{2 n -1}+q^n-q^{n-1}-\textbf{a}_1 q^{n-1}.\]
Since $\textbf{a}_1\geq 0$, this yields the desired upper bound on $|U|$, where equality implies that $\textbf{a}_1=0$, i.e. that no two elements of $U$ differ by a matrix of rank one. Hence, as discussed in the introduction, $\{A-B:B \in U\}$ is a constant rank-distance $2$ set for all $A\in U$.

\end{proof}

\section{Maximal partial spreads of $H(2n-1,q^2)$}
\label{sect:mps}

In this section we apply Theorem \ref{thm:boundher} to prove new results on the maximality of some partial spreads of $H(2n-1,q^2)$, and construct new maximal partial spreads in $H(3,q^2)$.

Thas \cite{Thas92} showed that spreads do not exist in $H(2n-1,q^2)$, and proved the upper bound for a partial spread set for $n$ even of $q^{2n-1}-q^n+q^{n-1}$. In \cite[Theorem 4.2]{DeBKlMetSt08}, this was improved to $q^{2 n-1}-q^{(3 n +1)/2}+q^{3n/2}-1$ for even $n\geq 4$, and $(q^3+q)/2$ for $n=2$. Much study has been dedicated to the spectrum of sizes of maximal partial spreads, see for example \cite{DeBKlMet12}.

As noted in Section \ref{sect:pre}, spreads in $W(2n-1,q)$ lead to partial spreads in $H(2n-1,q^2)$ of the same size. Such a spread always exists, and has order $q^n+1$. If $n$ is \emph{odd}, this is in fact the largest possible size of a partial spread in $H(2n-1,q^2)$:

\begin{theorem}\cite{Van09}
A partial spread in $H(2n-1,q^2)$, $n$ \emph{odd}, has size at most $q^n+1$.
\end{theorem}

This was proved by the fourth author using graph-theoretical techniques in \cite{Van09}, and again geometrically in \cite{Van11}. Note that this implies that a partial spread set in $H_n(\Fqt)$ has size at most $q^n$ when $n$ is odd, by Lemma \ref{lem:partialspreadcorrespondence}. We will see in Remark \ref{rem:spgen} that when $n$ is even, partial spreads of size larger than $q^n+1$ always exist. Theorem \ref{thm:boundher} implies that a {\it linear} partial spread set in $H_n(\Fqt)$ has size at most $q^n$ for any $n$.

Aguglia, Cossidente and Ebert \cite{AgCosEb01} proved the following (although their terminology is different).

\begin{theorem}[Aguglia-Cossidente-Ebert]
Any extension of a spread in $W(3,q)$ to a partial spread in $H(3,q^2)$ is maximal.
\end{theorem}

Theorem \ref{thm:boundher} above now gives the following new result.

\begin{theorem}
The extension of any \emph{semifield} spread in $W(2n-1,q)$ to a partial spread in $H(2n-1,q^2)$ is maximal.
\end{theorem}

\begin{proof}
Every semifield spread in $W(2n-1,q)$ is equivalent to a spread $D_U$, where $U$ is a linear spread set in $S_n(\Fq)$, and $|U|=q^n$. Consider $U$ now as a partial spread set in $H_n(\Fqt)$. Suppose there exists some $A \in H_n(\Fqt)$ such that $U \cup {A}$ is a partial spread set, $A \notin U$. Then $\det(A-B) \ne 0$ for all $B \in U$. But then $\det(\lambda A-B) \ne 0$ for all $B \in U$, $\lambda \in \Fp^{\times}$, and so $\langle A,U\rangle$ would be a linear partial spread set of size $pq^n$, contradicting Theorem \ref{thm:boundher}. Hence $U$ is a maximal partial spread set.
\end{proof}

The question remains open whether the extension of every (non-linear) symplectic spread in $W(2n-1,q)$ is a maximal partial spread in $H(2n-1,q^2)$ for $n$ even, $n>2$. Note that in the case $n$ odd, this was first proven for all spreads in $W(2n-1,q)$ in \cite{Luyckx08}, which was extended to all partial spreads of size $q^n+1$ in $H(2n-1,q^2)$ in \cite{Van09}. 

We now turn our attention to the question of the existence of an interval of integers such that, for each integer contained, there exists a maximal partial spread of that size. This question has received attention for the case of $\PG(3,q)$ (\cite{Hed2002}), $W(3,q)$ and $Q(4,q)$ (\cite{RoSt2010}, \cite{PeRoSt2010}, \cite{RotteyStorme2012}). We now construct maximal partial spreads in $H(3,q^2)$ for a range of sizes. Known results on the spectrum of sizes of maximal partial spreads in $H(3,q^2)$ can be found in for example \cite{DeBKlMet12}, \cite{CimFack05}. These partial spreads have received particular attention due to their equivalence with partial ovoids in the elliptic quadric $Q^{-}(5,q)$, which is the dual generalized quadrangle (see for instance \cite{PaTh09}). Though these new maximal partial spreads are not in general the largest nor smallest known, the authors know of no other constructions for an interval of sizes in this space.
\begin{theorem}
There exists a maximal partial spread in $H(3,q^2)$ of size $N$ for every integer $N$ in the interval $[q^2+1,q^2+q]$.
\end{theorem}

\begin{proof}
Let $\delta$ be some integer in $\{1,\ldots,q\}$. Choose some arbitrary subset $\Delta$ of $\Fq$ of size $\delta$, containing $0$.
Then define the set
\[
U_{\delta} = \left\{ \npmatrix{a&a\\a&0},\npmatrix{0&a\\a&\mu a}~:~a \in \Delta \right\} \cup \left\{ \npmatrix{0&\alpha\\\overline{\alpha}&0}~:~\alpha \in \Fqt \setminus \Delta \right\},
\]
where $\mu$ is chosen in $\Fq$ such that $x^2+y^2+(\mu-2)xy=0$ has no non-trivial solutions over $\Fq$. We claim that $U_{\delta}$ is a maximal partial spread set in $H_2(\Fqt)$, and so $D_{U_{\delta}}$ is a maximal partial spread in $H(3,q^2)$ and $|D_{U_{\delta}}| = |U_{\delta}|+1 = q^2+\delta$.

The fact that this is a partial spread set is easily verified. It remains to show that it is maximal. First note that 
\[
\{N(y-\alpha)~:~\alpha \in \Fqt\setminus\Delta\} = \left\{\begin{array}{cc} \Fq^{\times} &\textrm{if $y \in \Delta$}\\\Fq &\textrm{otherwise}. \end{array}\right.
\]
This is because, counting multiplicities, this set has size $q^2-\delta \geq q^2-q$. Clearly $0$ has multiplicity $0$ in the first case and $1$ in the second. If some non-zero $\lambda \in \Fq$ were not in this set, it could have size at most $(q-2)(q+1)+1 = q^2-q-1$, a contradiction.

Now suppose there exists some $\npmatrix{x&y\\\overline{y}&z}$ which extends $U_{\delta}$. Then $xz \notin \{N(y-\alpha):\alpha \in \Fqt\setminus\Delta\}$. By the previous argument, we must have $y \in \Delta$ and $xz = 0$. Therefore either $x=0$ or $z=0$. Suppose first $z=0$. Then 
\[
\npmatrix{x&y\\y&0} - \npmatrix{y&y\\y&0}
\]
is not invertible, a contradiction. Similarly in the case $x=0$ we get a contradiction, proving that such a matrix can't exist, thus proving maximality.

It is clear that $|U_{\delta}| = (q^2 - \delta)+ (2 \delta-1) = q^2+\delta-1$.
\end{proof}

Note that when $q=2$, $q^2+q=6$, which is the largest possible partial spread (\cite[Remark 4.4]{DeBKlMetSt08}.

We will see in the next section that there always exist partial spreads of $H(2n-1,q^2)$ of size greater than $q^n+1$ for all $n$ even.

\section{Large constant rank-distance sets in $H_n(\Fqt)$}
\label{sect:crs}

Constant rank-distance sets of rank less than $n$ have received far less attention than partial spread sets, with most of the focus applied to subspaces. See \cite{JSThesis} and the references therein. We saw in Theorem \ref{thm:oddgen} that if $k$ is odd, the maximum size of constant rank-distance $k$ sets in $H_n(\Fqt)$ is $q^k$, and this bound can be obtained by a linear set.

For even rank $k>0$, we can always find constant rank-distance sets of size larger than the largest known linear constant rank-distance set, due to the following construction. 

\begin{theorem}
\label{thm:crscon}
Suppose there exists a partial $(r-1)$-spread, $r\geq 1$, of $\PG(n-1,q^2)$ of size $N$. Then there exists a constant rank-distance $k = 2r$ set in $H_n(\Fqt)$ of size $N$.

\end{theorem}

\begin{proof}
Let $k=2r$. Let $D$ be a partial $(r-1)$-spread in $\PG(n-1,q^2)$ of size $N$. For each $S\in D$, choose some matrix $X_S \in M_{n \times r}$ whose column span is equal to $S$ (for example, by choosing a basis for $S$ and forming a matrix with these vectors as its columns). Next define
\[
A_S = X_S \overline{X_S}^T \in H_n(\Fqt).
\]
Finally define $U = \{A_S~:~S \in D\}$. We claim that $\rank(A_S-A_T)=k$ for all $S,T \in D$, $S \ne T$. For let $A_S,A_T \in U$. Then
\begin{align*}
A_S - A_T &= X_S \overline{X_S}^T - X_T \overline{X_T}^T\\
					&= \npmatrix{X_S & X_T} \npmatrix{I_r &0_r\\0_r&0_r} \npmatrix{\overline{X_S}^T \\\overline{X_T}^T} - \npmatrix{X_S & X_T} \npmatrix{0_r &0_r\\0_r&I_r} \npmatrix{\overline{X_S}^T \\\overline{X_T}^T}\\
					&= \npmatrix{X_S & X_T} \npmatrix{I_r &0_r\\0_r&-I_r} \npmatrix{\overline{X_S}^T \\\overline{X_T}^T}
\end{align*}
But $S$ and $T$ intersect trivially, and hence the matrix $\npmatrix{X_S & X_T}$ has rank $2r=k$. But $\npmatrix{I_r &0_r\\0_r&-I_r}$ also has rank $k$, and hence $A_S-A_T$ has rank $k$, as claimed. As shown in the introduction, this implies the existence of a constant rank-distance $k$ set in $H_n(\Fqt)$ of order $|U|=N$, proving the result.
\end{proof}

\begin{corollary}
\label{cor:crscon}
Suppose there exists a partial $(r-1)$-spread, $r\geq 1$, of $\PG(n-1,q^2)$ of size $N$. 
Then there exists a constant dimension $n$, constant distance $2k$ code in $H(2n-1,q^2)$ of size $N$.
\end{corollary}
\begin{proof} This follows immediately from Lemma \ref{lem:gendimcorrespondence} and Theorem \ref{thm:crscon}.
\end{proof}

\begin{remark}
In \cite[Theorem 4.2]{Beut75}, Beutelspacher showed that if $n=mr+d$, for non-negative integers $m,d$ with $0<d < r$ and $m \geq 2$, then there exists a partial of size 
\[
q^{r+d}\left(\frac{q^{(m-1)r}-1}{q^r-1}\right)+1 \geq q^{(m-1)r+d}+1 = q^{n-r}+1.
\]
If $n=mr$, $m \geq 2$, then by \cite{Segre64} there exists an $(r-1)$-spread of $\PG(n-1,q)$ of size
\[
\left(\frac{q^{mr}-1}{q^r-1}\right) \geq q^{n-r}+1.
\]
Hence there exists a constant rank-distance $k=2r$ set in $H_n(\Fqt)$ of size $q^{2(n-r)}+1 = q^{2n-k}+1$. This exceeds the size of the largest linear constant rank-distance $k$ set, which we saw in Theorem \ref{thm:boundher} has size at most $q^{2n-k}$.
\end{remark}

\begin{remark}\label{rem:magma}
Suppose $k=2$. The points of $\PG(n-1,q^2)$ form a $0$-spread of size $\frac{q^{2n}-1}{q^2-1}$. Hence there exists a constant rank-distance $2$ set of size $\frac{q^{2n}-1}{q^2-1}$ in $H_n(\Fqt)$ for all $q$ and all $n\geq 2$.

When $q=2$, $n=3$ this gives a constant rank-distance $2$ set of size 21. A computer calculation using the computer algebra package MAGMA \cite{MAGMA}, and independently using GAP \cite{GAP}, gave that the spectrum of sizes of maximal constant rank-distance $2$ sets in $H_3(\FF_{2^2})$ is $\{8,10,11,12,13,14,16,17,21\}$.  Hence the construction from Theorem \ref{thm:crscon} is maximal in this case. It is not clear whether this construction leads to maximal constant rank-distance sets in general.
\end{remark}

\begin{remark}
\label{rem:spgen}
Note that if $n=k=2r$, there exists an $(r-1)$-spread of $\PG(n-1,q^2)$, which has size $\frac{q^{2n}-1}{q^{2r}-1} = q^n+1$. Hence the construction from Theorem \ref{thm:crscon} gives a partial spread set in $H_n(\Fqt)$ of size $q^n+1$, and therefore  (by Lemma \ref{lem:partialspreadcorrespondence}) a partial spread of size $q^n+2$ in $H(2n-1,q^2)$, which is larger than the largest possible linear partial spread set.
\end{remark}

\section*{Acknowledgements}
The research of the second and the last author is supported by the Research Foundation Flanders-Belgium (FWO-Vlaanderen).
The research of the third author is supported by a Progetto di Ateneo from Universit\`a di Padova (CPDA113797/11).
\bibliographystyle{plain}

\end{document}